\documentclass[10pt,a4paper]{amsart}
\usepackage{amscd,amssymb,amsopn,amsmath,amsthm,graphics,amsfonts,enumerate,verbatim,calc}
\usepackage[dvips]{graphicx}
\usepackage[utf8]{inputenc}

\usepackage{mathpazo}
\usepackage{color}
\usepackage{latexsym}
\usepackage{amsthm,amsfonts,amssymb,mathrsfs}
\usepackage{rotating}
\usepackage[leqno]{amsmath}
\usepackage{xspace}
\usepackage[all]{xy}
\usepackage{longtable}
\headsep=5 mm \numberwithin{equation}{section}
\newtheorem{theorem}{Theorem}[section]
\newtheorem{lemma}[theorem]{Lemma}
\newtheorem{notation}[theorem]{Notation}
\newtheorem{proposition}[theorem]{Proposition}
\newtheorem{corollary}[theorem]{Corollary}

\newtheorem{problem}[theorem]{Problem}
\theoremstyle{definition}

\theoremstyle{remark}
\newtheorem{remark}[theorem]{Remark}
\newtheorem{fact}[theorem]{Fact}
\newtheorem{example}[theorem]{Example}
\newtheorem{observation}[theorem]{Observation}
\newtheorem{discussion}[theorem]{Discussion}
\newtheorem{question}[theorem]{Question}

\newtheorem{acknowledgement}{Acknowledgement}

\newcommand{\Ass}{\operatorname{Ass}}
\newcommand{\type}{\operatorname{r}}

\newcommand{\Soc}{\operatorname{Soc}}

\newcommand{\BNSI}{\operatorname{BNSI}}

\newcommand{\syz}{\operatorname{Syz}}
\newcommand{\Spec}{\operatorname{Spec}}

\newcommand{\Fitt}{\operatorname{Fitt}}

\newcommand{\Ht}{\operatorname{ht}}

\newcommand{\pd}{\operatorname{p.dim}}

\newcommand{\m}{\operatorname{m}}

\newcommand{\Syz}{\operatorname{Syz}}

\newcommand{\UFD}{\operatorname{UFD}}

\newcommand{\HH}{\operatorname{H}}

\newcommand{\rank}{\operatorname{rank}}

\newcommand{\e}{\operatorname{e}}

\newcommand{\V}{\operatorname{V}}

\newcommand{\Ext}{\operatorname{Ext}}
\newcommand{\Ord}{\operatorname{Ord}}
\newcommand{\tr}{\operatorname{tr}}

\newcommand{\Supp}{\operatorname{Supp}}

\newcommand{\Tor}{\operatorname{Tor}}

\newcommand{\Hom}{\operatorname{Hom}}

\newcommand{\Ann}{\operatorname{Ann}}

\newcommand{\depth}{\operatorname{depth}}

\newcommand{\len}{\ell}

\newcommand{\lo}{\longrightarrow}
\newcommand{\fm}{\mathfrak{m}}
\newcommand{\fp}{\frak{p}}

\newcommand{\fn}{\frak{n}}

\newcommand{\n}{\mathfrak n}

\begin{document}
	
	\author[]{Mohsen Asgharzadeh}
	
	\address{}
	\email{mohsenasgharzadeh@gmail.com}

	\title[ ]
	{On the initial Betti numbers}

	
	\begin{abstract}
	 Let $R$ be a  Cohen-Macaulay local ring  possessing a canonical module.
	We compare the initial and terminal Betti numbers of   modules  in a series of nontrivial cases. We pay
 special attention to the Betti numbers of the canonical module.  Also,
	 we compute $\beta_0(\omega_{\frac{R}{I}})$ in some cases, where $I$ is a product of two ideals.
	\end{abstract}

\maketitle

\section{Introduction}
	Throughout this note, let $(R,\fm,k)$ be a Cohen--Macaulay local ring with canonical module $\omega_R$. 
	We may assume that $R$ is not Gorenstein, as otherwise the situation trivializes. 
	The study of syzygy modules has a long history in commutative algebra; 
	nevertheless, our knowledge remains rather limited, even in the case of the initial syzygy modules of $\omega_R$. 
	
	\begin{problem}
		When is $\beta_n(\omega_R)>\beta_{n-1}(\omega_R)$? 
		How can one determine $\beta_n(\omega_R)$ from the data $\{\beta_{i}(\omega_R): i<n\}$?
	\end{problem}
	
	The case $n=1$ of this problem was already posed in \cite[Question~2.6]{growth}. 
	Some partial progress has been made: it was shown in \cite{csv} that $\beta_1(\omega_R)\geq\beta_{0}(\omega_R)$, 
	while in \cite{av} it is proved that $\beta_{n-1}(\omega_R)<\beta_{n}(\omega_R)$ for Cohen--Macaulay rings of embedding codepth at most $3$. 
	Furthermore, over Cohen--Macaulay rings of minimal multiplicity, the sequence $\{\beta_i(\omega_R)\}$ can be computed explicitly, 
	see \cite[10.8.2]{avv}. 
	
	We present several observations related to the above problem, 
	focusing on lower bounds for the growth of Betti numbers of canonical modules, 
	and applications to the structure of certain classes of Cohen--Macaulay rings. 
	For clarity, we outline our main contributions below:
	
	\medskip
	\noindent 
	\textbf{Observation A.} For each odd $n$, one has
	\(
	\beta_n(\omega_R)\geq \sum_{i=0}^{n-1}(-1)^{n-i+1}\beta_i(\omega_R).
	\)
	If $n$ is even, then
	\(
	\beta_n(\omega_R)\geq \sum_{i=0}^{n-1}(-1)^{n-i+1}\beta_i(\omega_R)+2.
	\)
	In particular, if $R$ is of type $2$, then $\beta_2(\omega_R)\geq \beta_1(\omega_R)$.
	
	\medskip
	\noindent 
	\textbf{Observation B.} Let $R_{d,n}:=k[x_1,\ldots,x_d]/\fn^n$. 
	If $M$ is nonfree, then the inequality
	\[
	\beta_{i+1}(M)-\beta_i(M)\geq d-1 \qquad \forall i\geq 2
	\]
	holds for all $n>1$, strengthening the result of Gover and Ramras.
	We then extend this to more general setting $R:=S/\n^n$, where $(S,\n)$ is a regular local ring of dimension $d\ge2$ and
	$n\ge2$.
	\\
\textbf{Observation C.} One has
	\(
	\beta_{1}(\omega_{R_{d,n}})-\beta_0(\omega_{R_{d,n}})
	= \operatorname{rank}_{k[\underline{x}]}(\Syz_{d-1}(R_{d,n}))\geq d-1.
	\)

If we impose certain restrictions on the parameters $d$ and $n$, the previous inequality  can be refined. 
In Observation~\ref{2} and Proposition~\ref{sh}, we show that the inequality is far from being sharp by presenting a stronger bound. 
Moreover, Example~\ref{sh1e} demonstrates that this new bound is indeed sharp, but fails in the case $i=1$. 
Additional information on the Betti numbers of canonical modules over rings of the form $A/\fn^n$ for small values of $n$ 
is recorded in Corollary~\ref{mor1}.  We close Section 4 by computing Betti numbers of certain fiber products. 

\medskip

We next turn to applications of the initial Betti numbers of canonical modules. 
Let $I\subset S:=k[\underline{x}]$ be a Cohen--Macaulay homogeneous non-principal ideal with a $q$-linear resolution for some $q>1$, and set $R:=S/I$. 
In analogy with Observation~C, one can show that
\(
\beta_1(\omega_R)- \beta_0(\omega_R)\geq \Ht(I)-1.
\)
In particular, $\beta_1(\omega_R)>\beta_0(\omega_R)$, which ensures that $R$ is not Gorenstein. 
As an application, we discuss the role of such inequalities in the existence of semidualizing modules (see \S5). 

\medskip

Another classical motivation stems from work of Huneke~\cite{h}. 
Suppose $S$ is an unramified regular local ring and $J,K\subset S$ are proper ideals of height at least two. 
Using a positivity property of higher Euler characteristics, Huneke proved that $S/JK$ is never Gorenstein. 
This settled a question originally posed by Eisenbud and Herzog in the more general context of regular local rings. 
Here we establish another situation in which this problem has a positive answer:

\medskip
\noindent 
\textbf{Observation D.} 
Let $S$ be a regular ring and let $J,K\subset S$ be ideals with $\Ht(JK)=2$. 
Then $S/JK$ is not Gorenstein. 

\medskip

If in addition the ideal $JK$ in Observation~D is Cohen--Macaulay, one can see that
\(
\beta_0(\omega_{S/JK})=\mu(JK)-1\geq 2.
\)
More generally, let $(A,\fm,k)$ be a local ring containing $k$, and let $J,K\subset A$ be ideals of height $\ell$ that are monomial with respect to a regular sequence of length $\ell>1$. 
In this case we prove that
\(
\type(A/KJ) \geq 2\,\type(A),
\)
so in particular $A/KJ$ is not Gorenstein. 

\medskip

Huneke~\cite{h} also posed a further question in this context:

\begin{question}
	Let $S$ be a regular local ring and let $J,K\subset S$ be proper ideals such that $JK$ is Cohen--Macaulay. 
	Is it true that
	\(
	\beta_0(\omega_{S/JK})\geq \Ht(JK)?
	\)
\end{question}

We show:

\medskip
\noindent 
\textbf{Observation E}. 
Let $J$ be a radical ideal  of a noetherian local ring $A$. Let $n>1$ be an integer. 
	\begin{enumerate}
	\item[i)]
	If $\ell:=\Ht(J)>1$, then  	$A/ J^n$ is not Gorenstein.	
	\item[ii)] If $A/ J^n$ is Cohen-Macaulay, then  $ \beta_0(\omega_{\hat{A} /\hat{ J}^n})\geq \ell=\Ht(J).$ 
	\item[iii)] Assume  in addition to ii) that $A$ is regular, then  $\beta_0(\omega_{\frac{A} {J ^n}})\geq {n+\ell-2 \choose \ell-1}.$
\end{enumerate} 
\medskip
\noindent 

Let $(A,\fm,k)$ be a local ring containing $k$,  
$J$ be a complete-intersection ideal and  $n>1$. It is easy to see that the inequality
$\beta_0(\omega_{\frac{ \hat{A} }{\hat{ J}^n}})\geq \Ht(J)$ is valid.   
One  may like to state this observation in the following stronger format:

\medskip
\noindent
\textbf{Observation F}. Let $(A,\fm,k)$ be a $d$-dimensional  local ring containing  $k$
and $J$ be a generically complete-intersection ideal of height $\ell>1$  and assume 
$\ell<d$.   Then 	$A/ J^n$ is not Gorenstein for all $n>1$.
Suppose in addition	$A/J^n$ is Cohen-Macaulay, then 	$\beta_0(\omega_{\frac{ \hat{A} }{\hat{ J}^n}})\geq \Ht(J)$.
 	\medskip
 	\noindent
 	
 	 As another sample, we show:
 	 
\medskip
\noindent 	 
\textbf{Observation G}. 
 	Let   $A$ be    Cohen-Macaulay  and  $I= JK$ for two proper ideals $J, K$    
 	such that $\frac{A}{I}$ is Cohen-Macaulay and  
 of minimal   multiplicity. Then $ \beta_0(\omega_{\frac{\hat{A}}{\hat{I}}})\geq \Ht(I)$. The equality implies that  $A$ is regular.
\medskip
\noindent

\section{ Comparison on $\beta_i(\omega_R)$: The general case}

In this section  $(R,\fm,k)$  is a  Cohen-Macaulay local ring which is not Gorenstein. Also, we assume 
$R$ equipped with a canonical module.
 The notation $\e(-)$ stands for the Hilbert-Samuel multiplicity.  It is additive with respect to short exact sequence of maximal Cohen-Macaulay modules.  The $i^{th}$  Betti number  of $M$ is given by $\beta_i(M):=\dim_k(\Tor^R_i(k,M))$.  A minimal free resolution of $M$ is of  the form
$$\cdots \lo R^{\beta_{i-1}(M)}\stackrel{f_{i-1}}\lo  R^{\beta_{i-2}(M)}\lo\cdots\lo  R^{\beta_{0}(M)}\lo M\lo 0.$$
The $i^{th}$ \textit{syzygy} module of $M$ is $\Syz_i(M) := \ker(f_{i-1})$ for all $i>0$.
We denote by $\mu(-)$  the minimal number of   generators of a module $(-)$.
\begin{lemma}\label{esyz}
 One has $\e(\Syz_n(\omega_R))=\e(R)(\sum_{i=0}^{n-1}(-1)^{n-1-i}\beta_i+(-1)^n).$ 
\end{lemma}

\begin{proof}
	We proceed   by induction on $n$. First, we deal with the case $n=1$. By definition, there is an  exact sequence $0\to\Syz_1(\omega_R)\to R^{\beta_0(\omega_R)}\to\omega_R\to 0.$
Recall that $\e (\omega_R)=\e(R)$.  From this, $ \e(\Syz_1(\omega_R))=\e(R)(\beta_0(\omega_R)-1)$.
	This completes the argument when  $n=1$.
	By induction, $$\e(\Syz_n(\omega_R))=\e(R)(\sum_{i=0}^{n-1}(-1)^{n-i}\beta_i+(-1)^n).$$ We look at the exact sequence $0\to\Syz_{n+1}(\omega_R)\to R^{\beta_n(\omega_R)}\to\Syz_n(\omega_R)\to 0.$ 
	From this, $$\e(\Syz_{n+1}(\omega_R))=\e(R) \beta_n(\omega_R)- \e(R)(\sum_{i=0}^{n-1}(-1)^{n-1-i}\beta_i+(-1)^n)= \e(R)(\sum_{i=0}^{n}(-1)^{n-i}\beta_i+(-1)^{n+1}).$$
\end{proof}
Here, $\ell(-)$ stands for the length function.
\begin{proposition}\label{g}  	Let   $M:=K$ be of finite length.
	\begin{enumerate}
		\item[i)]   If $n$ is even, then 	$\sum_{i=0}^n(-1)^{i}\ell(\Tor^R_i(M,\omega_R))\geq 2\ell(M)$.
		\item[ii)]	 If $n$ is odd, then	$\sum_{i=0}^n(-1)^{i}\ell(\Tor^R_i(M,\omega_R))\geq 0$.
	\end{enumerate} 
\end{proposition}

\begin{proof}
	We have the exact sequence $0\to\Syz_{n+1}(\omega_R)\to R^{\beta_n(\omega_R)}\to\Syz_n(\omega_R)\to 0.$
	Note that $\Syz_+(\omega_R)$ is nonzero and maximal Cohen-Macaulay. Then $$\e(\Syz_n(\omega_R))=\e(R^{\beta_n(\omega_R)})-\e(\Syz_{n+1}(\omega_R))<\e(R^{\beta_n(\omega_R)})=\beta_n(\omega_R)\e(R)\quad(+)$$
	Now we look at   $0\to\Syz_{n }(\omega_R)\to R^{\beta_{n-1}(\omega_R)}\to\Syz_{n-1}(\omega_R)\to 0 $.
	Putting these together, we have
	\[\begin{array}{ll}
	{\beta_{n-1}}(\omega_R)\e(R) &=\e(R^{\beta_{n-1}(\omega_R)})\\
	&= \e(\Syz_n(\omega_R))+\e(\Syz_{n-1}(\omega_R))\\
	&\stackrel{(+)}<\beta_n(\omega_R)\e(R)+\e(\Syz_{n-1}(\omega_R))\\
	&\stackrel{\ref{esyz}}=\beta_n(\omega_R)\e(R)+\e(R)(\sum_{i=0}^{n-2}(-1)^{n-2-i}\beta_i+(-1)^{n-1}).
	\end{array}\]From this we  deduce   the following:
	\begin{enumerate}
		\item[i)]    If $n$ is even, then $\beta_n(\omega_R)\geq \sum_{i=0}^{n-1}(-1)^{n-i+1}\beta_i(\omega_R)+2$.
		\item[ii)]	 if $n$ is odd, then $\beta_n(\omega_R)\geq \sum_{i=0}^{n-1}(-1)^{n-i+1}\beta_i(\omega_R)$.
	\end{enumerate}
	This completes the proof.
\end{proof}

\begin{question}\label{Qg}  	Let   $M$ be of finite length. One may ask
is then {Proposition} \ref{g} hold  for $M$?
\end{question}

We though for some times that the answer to {Question} \ref{Qg} is positive.
In order to present a negative answer, 
for an Artinian local $k$-algebra $R$, write
$D=\operatorname{Hom}_k(R,k)=\omega_R$ and set
\[
\chi_1(M,\omega_{R}):=\len(M\otimes_R \omega_{R})-\len(\Tor_1^R(M,\omega_{R})).
\]

\begin{fact}\label{prop:squarezero}
	Suppose $\m^2=0$ and $e=\dim_k\m\ge2$. Then
	\(
	\beta_0^R(\omega_{R})=e\), \(
	\beta_i^R(\omega_{R})=e^{i-1}(e^2-1)	\) for   \(i\ge1\).
	Moreover, for every finite $R$-module $M$,
	\(
	{\ \chi_1(M,\omega_{R})
		=e\len(M)-(e^2-1)\beta_0^R(M).\ }
	\)
	Consequently, for $M\ne0$, one has
	\(
	\chi_1(M,\omega_{R})>0
	\) iff
	\(	\frac{\len(M)}{\beta_0^R(M)}>e-\frac1e.
	\)
\end{fact}
\begin{proof}
	Choose a basis $x_1,\ldots,x_e$ of $\m$ and the dual basis
	$u_0,u_1,\ldots,u_e$ of $\omega_{R}$, with $u_0$ dual to $1$.
	Then $x_i u_j=\delta_{ij}u_0$ for $j\ge1$ and $x_i u_0=0$.
	Thus $\omega_{R}$ is minimally generated by $u_1,\ldots,u_e$.
	In the corresponding cover $R^e\longrightarrow \omega_{R}$, the kernel is
	contained in $\m R^e$, is annihilated by $\m$, and has length
	\[
	e\len(R)-\len(\omega_{R})=e(e+1)-(e+1)=e^2-1.
	\]
	Therefore $\syz_1^R(\omega_{R})\cong k^{e^2-1}$.
	Since $\syz_1^R(k)=\m\cong k^e$, all higher Betti numbers follow.
	Tensor the  exact sequence
	$0\longrightarrow k^{e^2-1}\longrightarrow R^e
	\longrightarrow \omega_{R}\longrightarrow0$ with $M$, yields
	\[
	0\longrightarrow\Tor_1^R(M,\omega_{R})\longrightarrow
	(M/\m M)^{e^2-1}\longrightarrow M^e
	\longrightarrow M\otimes_R \omega_{R}\longrightarrow0.
	\]
	Taking lengths proves the identity and its positivity criterion.
\end{proof}

\begin{example}\label{cor:sign}
	For $R=k[x,y]/(x,y)^2$ one has
$\chi_1(k,\omega_R)=-1,$
$\chi_1(\omega_R,\omega_R)=0$ and 
$\chi_1(R,\omega_R)=3.$
	Hence all three signs occur. In particular,
{Question} \ref{Qg} is not true.
\end{example}

Let M be a  module of depth $d$. We set
$\type(M) := \dim_k \Ext^d_R(k,M)$, and we call it  the type of $M$.

\begin{corollary}\label{t2}Let $R$ be of type $2$. Then $\beta_2(\omega_R)\geq \beta_1(\omega_R)$.
\end{corollary}

\begin{proof}
	In the light of Proposition \ref{g} we have  $\beta_2(\omega_R)\geq \beta_1(\omega_R)-\beta_0(\omega_R)+2$. By the assumption,
	$\beta_0(\omega_R)=2$. From this, $\beta_2(\omega_R)\geq \beta_1(\omega_R)$.
\end{proof}
\begin{question}What is $\inf\{\beta_2(\omega_R):R\emph{ is  Cohen-Macaulay and not Gorenstein}\}$? \end{question}This invariant lives in the interval $[2,6].$

\begin{example}
	Let  $R:=k[[t^3,t^4,t^5]]$. Then $R$ is Cohen-Macaulay, generically Gorenstein and of type $2$. Also, $\beta_i(\omega_{R})=2^{i-1}3$ for all $i>0$.
\end{example}

\begin{proof}Betti numbers of the canonical module are not change under reduction
	by a regular sequence.
In view of  $R/t^3R\cong\frac{k[X,Y]}{\fm^2}$ we get the claim (see e.g. Example \ref{ave}). 
	\end{proof}

Recall that the first part of {Fact} \ref{prop:squarezero} were first appeared in  \cite[10.8]{avv}. Here, we reprove it by a direct  computation. 

\begin{example}\label{ave}
	Let $(A,\fm,k)$ be a local ring of embedding dimension $e$.
	Let  $R:=\frac{A}{\fm^2}$.  Then $\beta_0(\omega_{R})=e$ and $\beta_i(\omega_{R})=e^{i-1}(e^2-1)$ for all $i>0$.
\end{example}

\begin{proof} Recall that $\type(R)=\dim\frac{(0:\fm)}{\fm^2}=\mu(\fm)=e$ and that
	 $\ell(R)=\ell(\m)+1=\ell(\fm/ \fm^2)+1=e+1$. By Matlis duality, $\ell(\omega_{R})=\ell(R)$. By $0\to\Syz_1(\omega_{R})\to R^e\to\omega_{R}\to 0$
	we see   $$\beta_1(\omega_{R})=\mu(\Syz_1(\omega_{R}))=\dim\frac{\Syz_1(\omega_{R})}{\fm\Syz_1(\omega_{R})}=\ell(\Syz_1(\omega_{R}))=\ell(R^e)-\ell(\omega_{R})=e(e+1)-\ell(\omega_{R})=e^2-1.$$
	By repeating  this for $0\to\Syz_i(\omega_{R})\to R^{\beta_i(\omega_{R})}\to\Syz_{i-1}(\omega_{R})\to 0$, the claim follows.
\end{proof}

\section{ Comparison on  $\beta_i(\omega_R)$:  Some specialized cases}

Let $d>1$ and $n>1$ be integers.
We set $S:=k[X_1,\ldots,X_d]$, $\fn:=(X_1,\ldots,X_d)$ and $R:=S/\fn^n$.  
Recall from \cite{ram2} that
 $\beta_{i+1}(M)-\beta_i(M)\geq d-1$ for all $i\geq 2$,   $n\gg 0$   and  all nonfree module $M$. This 
is false for $i=0$. Namely, $\beta_{1}(R/x_1R)=\beta_0(R/x_1R)=1.$
 What can say on the validity of   $\beta_{1}(-)-\beta_0(-)\geq d-1$?  The following  result answers this:
 
\begin{proposition}\label{d-1}
We have   $\beta_{1}(\omega_R)-\beta_0(\omega_R)= \rank_S(\Syz_{d-1}(R))\geq d-1.$ 
\end{proposition}

\begin{proof} 
Recall
that powers of $\fn$ are
equipped with  linear resolutions. We look at the minimal free resolution of  $R$
over $S$:$$0\lo S^{n_d}\stackrel{\phi}\lo S^{n_{d-1}}\lo\ldots\lo S\lo R\lo 0\quad(\ast)$$
In particular, the components of   $\phi$  are linear with respect to fixed bases of $S^{n_d}$ and $S^{n_{d-1}}$. 	We bring the following claim:
\begin{enumerate}
	\item[Claim i)]  We have $ n_{d-1}-n_d=\rank_S(\Syz_{d-1}(R))\geq d-1$.	
\end{enumerate} 
Indeed, we have the following exact sequence of $S$-modules: $$0\lo S^{n_d}\lo S^{n_{d-1}}\lo\Syz_{d-1}(R)\lo 0\quad(+)$$
As $\ell_S(R)<\infty$, we see $\depth_S(R)=0$. Due to Auslander-Buchsbaum formula we know that $\Syz_{d-1}(R)$ is not free.
	In the light of  \cite[Corollary 1.7]{evan}  we see $\rank(\Syz_{d-1}(R))\geq d-1$.  
	By applying the additivity of $\rank$ over $(+)$ we deduce that  $ n_{d-1}-n_d\geq d-1$.
 
We apply $ \Hom_S(-,S)$
to $(\ast)$ we arrive to the following complex:$$\ldots\lo S^{n_{d-1}}\stackrel{\phi^t}\lo S^{n_d}\lo \omega_R\lo 0 \quad(\dagger)$$
which is exact at the right. Recall that $\fn^n\omega=0$.  By  applying $-\otimes _SR$ to $(\dagger)$  we get to the  
 presentation  $R^{n_{d-1}}\stackrel{\overline{\phi^{\tr}}}\lo R^{n_d}\to \omega_R\to 0.$  Since $\phi$ is of
  linear type, its components is not in $\fn^i$ for any $i>1$.
In particular, no column or row of $\overline{\phi^{\tr}}$ is zero. Thus
the displayed presentation is minimal, e.g., $\beta_0(\omega_R) = n_d$
and  $\beta_1(\omega_R) = n_{d-1}$. From this, we have
 $(d-1)+ \beta_0(\omega_R)
\stackrel{i)}\leq\rank_S(\Syz_{d-1}(R))+ n_d 
\stackrel{i)}=\beta_1(\omega_R).$  
\end{proof}

\begin{example}
By Example \ref{ave} the above bound may be achieved.
\end{example}

\begin{corollary}
	Let $R=k[X_1,\ldots,X_d]/ \fn^n$ where  $d,n>1$. Then  $\beta_1 (\omega_R)\geq (d-1)+
	{d+n-2 \choose n-1}$. In particular, $\beta_1 (\omega_R)\geq 2d-1$.
\end{corollary}

\begin{proof}  Recall that $
	\beta_0(\omega_R) = \mu(\omega_R) =\dim_k(0:_R\fm) =\dim_k(\fm^{n-1}) ={d+n-2 \choose d-1}  \geq d.$ 
Since $\beta_1(\omega_R)-\beta_0(\omega_R)\geq d-1$,  we deduce $\beta_1 (\omega_R)\geq 2d-1$.
\end{proof}

The particular bound is not sharp:

\begin{corollary}
	Let $R=k[X_1,\ldots,X_d]/\fn^n$ where  $n>2$ and   $d>1$. Then  $\beta_1 (\omega_R)\geq 2d+1$.
\end{corollary}

\begin{proof}We combine
	$\beta_0(\omega_R)  ={d+n-2 \choose n-1}  \geq d+2$
	along with $\beta_1(\omega_R)-\beta_0(\omega_R)\geq d-1$ to deduce that $\beta_1 (\omega_R)\geq 2d+1$.
\end{proof}

\begin{proposition}\label{ph3}
	Let $I$ be a   non-principal Cohen-Macaulay homogeneous ideal of $S:=k[x_1,\ldots,x_d]$ with $q$-linear resolution
	for some $q>1$. Set $R:=S/I$. Then $\beta_1(\omega_R)- \beta_0(\omega_R)=\rank_S(\Syz_{h-1}(R))\geq \Ht(I)-1>0.$ 
	In particular,  $R$ is not Gorenstein.
\end{proposition}

\begin{proof}Suppose on the way of contradiction that $h:=\Ht(I)=1$. Due to the Cohen-Macaulay assumption, $I$ is unmixed. Over $\UFD$, unmixed and height-one ideals are  principal. Hence, $\mu(I)=1$.
	This is excluded by the assumption. 
	We  may assume that $h>1$.  	Since $I$ is perfect,
	$\pd(S/I)=\Ht(I)$.  Recall that $I$ has 
$q$-linear resolution means that the graded minimal free resolution
of $R$ is of the form
$$0 \lo S(-q - h)^{\beta_h(R)} \stackrel{\phi}\lo \ldots \lo S(-q - 1)^{\beta_2(R)} \lo S(-q)^{\beta_1(R)} \lo S\lo R \lo 0.$$In particular, components of $\phi:=(a_{ij})$ are linear with respect to fixed bases of $S(-q - h)^{\beta_h(R)}$ and $S(-q - h+1)^{\beta_{h-1}(R)}$. We look at $0\to S^{\beta_h(R)}\to S^{\beta_{h-1}(R)}\to\Syz_{h-1}(R)\to 0 $ and recall from  \cite[Corollary 1.7]{evan} that $\rank_S(\Syz_{h-1}(R))\geq h-1\geq 1$. 
Since $q>1$, no column or row of $\overline{\phi^{\tr}}:=\phi^{\tr}\otimes_SR$ is zero. Similar to
Proposition \ref{d-1}, the presentation
  $R^{\beta_{h-1}(R) }\stackrel{\overline{\phi^{\tr}}}\lo R^{ \beta_h(R)}\to \omega_R\to 0 $   is minimal. So, $\beta_1(\omega_R)- \beta_0(\omega_R)=\beta_{h-1}(R)-\beta_h(R)= \rank_S(\Syz_{h-1}(R))\geq h-1.$
\end{proof}

\begin{example}\label{3e}
	The first item shows that the previous bound may be achieved.
	The second  item shows that  the inequality is not an equality.

 i)  Let $I:=(x,y)^2$  and $S:=k[x,y]$. The minimal free resolution of $I$ is  $0\to S^2(-3)\to S^3(-2)\to  I\to 0 $.
Thus,  $I$ is  $2$-linear 
and   $\beta_1(\omega_R)=3=2+1= \beta_0(\omega_R)+(\Ht(I)-1).$

 ii) 	
Let $I:=(x,y,z)^2$  and $S:=k[x,y,z]$. The minimal free resolution of $I$ is 
$0\to S^3(-4)\to S^8(-3)\to S^6(-2)\to I\to 0.$ 
Thus,  $I$ is  $2$-linear  
and    $\beta_1(\omega_R)=8>3+2= \beta_0(\omega_R)+(\Ht(I)-1).$ 
\end{example}
 
 \begin{notation}
For $a,b\ge1$, we set
\(	A_a =k[x_1,\ldots,x_a]/(x_1^2,\ldots,x_a^2) \),
\(	B_b =k[y_1,\ldots,y_b]/(y_1^2,\ldots,y_b^2)\)
and consider to their fiber product $	R_{a,b}:=A_a\times_k B_b\cong
	\frac{k[x_1,\ldots,x_a,y_1,\ldots,y_b]}
	{(x_i^2,\ y_j^2,\ x_i y_j:1\le i\le a,\ 1\le j\le b)}.$ 
 The ring $R_{a,b}$
is Artinian local, with canonical module
$D_{a,b}:=\omega_{R_{a,b}}=\operatorname{Hom}_k(R_{a,b},k)$.\end{notation}

\begin{proposition}\label{prop:fiber}
	Writing $e=a+b$, one has
	\(
	\beta_0^{R_{a,b}}(D_{a,b})=2\), 
\(	\beta_1^{R_{a,b}}(D_{a,b})=e+1,\)
	and
	\[
	{\ \beta_2^{R_{a,b}}(D_{a,b})
		=e(e+1)-\binom a2-\binom b2.\ }
	\]
\end{proposition}
\begin{proof}
	Put $X=x_1\cdots x_a$ and $Y=y_1\cdots y_b$.
	The socle of $R=R_{a,b}$ is $kX\oplus kY$, so the type is two.
	Let $u,v\in D:=D_{a,b}$ be the dual functionals of $X,Y$ in the squarefree
	monomial basis. The module $D$ is generated by $u,v$; multiplication
	by complementary squarefree monomials obtains the duals of all basis
	monomials, including the constant monomial.
	The map $R^2\to D$ sending the two basis vectors to $u,v$
	has kernel generated by
$\{ y_j\varepsilon_1 , x_i\varepsilon_2, w:=X\varepsilon_1-Y\varepsilon_2\ : 1\le j\le b,\  1\le i\le a  \}. $
	Indeed, their generated submodule is the direct sum
	\[
	\n_B\varepsilon_1\oplus\n_A\varepsilon_2\oplus kw
	\cong\n_B\oplus\n_A\oplus k,
	\tag{\(*\)}\label{eq:firstsyz}
	\]
	where $\n_A=(x_1,\ldots,x_a)A_a$ and
	$\n_B=(y_1,\ldots,y_b)B_b$, inflated to $R$-modules.
	The last summand is annihilated by the maximal ideal.
	These summands are independent: their supports in the two coordinates
	are respectively the positive $y$-monomials, the positive $x$-monomials,
	and the pair $(X,-Y)$. Their total length is
	\(
	(2^b-1)+(2^a-1)+1=\len(R).
	\)
	The kernel also has length
	$2\len(R)-\len(D)=\len(R)$, so this submodule is the entire kernel.
	Its minimal number of generators is $b+a+1=e+1$.
	
	We next compute $\beta_1^R(\n_A)$.
	In the minimal surjection $R^a\to\n_A$, a relation splits
	uniquely into an $A_a$-relation and a vector in $\n_B^a$.
	The constants in the coefficients of an $A_a$-relation all vanish,
	since the $x_i$ are independent modulo $\n_A^2$.
	It follows that the kernel splits as
	\(
	\ker(A_a^a\to\n_A)\oplus\n_B^a.
	\)
	The first summand has exactly $\binom{a+1}{2}$ minimal generators:
	they are $x_i\varepsilon_i$ and
	$x_j\varepsilon_i-x_i\varepsilon_j$ for $i<j$.
	To see generation, lift a relation to the polynomial ring. Its image
	belongs to $(x_1^2,\ldots,x_a^2)$; subtract multiples of
	$x_i\varepsilon_i$ and use the Koszul relations on $x_1,\ldots,x_a$.
	The displayed linear relations are independent modulo the maximal
	ideal times the relation module, proving minimality. Thus
	\[
	\beta_1^R(\n_A)=\binom{a+1}{2}+ab,
	\qquad
	\beta_1^R(\n_B)=\binom{b+1}{2}+ab.
	\]
	Also $\beta_1^R(k)=e$. Applying these identities to
	\eqref{eq:firstsyz} gives
	\[
	\begin{aligned}
		\beta_2^R(D)
 =\binom{a+1}{2}+\binom{b+1}{2}+2ab+e
	 =e(e+1)-\binom a2-\binom b2.
	\end{aligned}
	\]
\end{proof}

\section{Mores  on the formula by Gover and Ramras}

\subsection{Rings containing a field}
Let $M$ be nonfree over $R_{d,n}$. Assume $d>1$ and  $n$ is large enough.
Recall that Gover and Ramras
proved  that  $\beta_{i+1}(M)-\beta_i(M)\geq d-1$   for all $i\geq 2$.
This brings up some natural questions. The first one posted by Gover and Ramras: 
how much $n$ should be large?

\begin{observation}\label{2}
	Let  $M$  be nonfree over one of the following rings:
		\begin{enumerate}
		\item[i)] $R:= k[X_1,\ldots,X_d]/ \fn^{2}$ with $d>1$
		\item[ii)]$R:=k[X_1,X_2]/(X_1,X_2)^{n}$  with $n>1$.
	\end{enumerate}  
	 Then   $\beta_{i+1}(M)-\beta_{i}(M)\geq d-1$ for all $i\geq 1$.	In particular, $\beta_{i+1}(\omega_R)-\beta_i(\omega_R)\geq d-1$ for all $i\geq 0$.
\end{observation}

\begin{proof}
i)
	 Let $i\geq1$. In view of \cite[Proposition 2.4]{fun} we have  
	  $$\beta_{i+1}(M)\geq (\dim_k\frac{\fm}{\fm^2})\beta_{i}(M)=\mu(\fm)\beta_i(M)=d\beta_i(M)=(d-1)\beta_i(M)+\beta_i(M)\geq(d-1)+\beta_i(M).$$
	  ii)	By a result of Ramras $\{\beta_{j}(M)\}_{j\geq1}$ is strictly increasing. Since $d-1=1$, this yields the claim.
\end{proof}

\begin{proposition}\label{ob}
	Let $R=k[X_1,\ldots,X_d]/\fn^n$ with $d>1$ and $n>1$. If $M$  is nonfree, then   $\beta_{i+1}(M)-\beta_{i}(M)\geq d-1$ for all  $i\geq 2$.
	In particular, $\beta_{i+1}(\omega_R)-\beta_i(\omega_R)\geq d-1$ for all $i\neq  1$.
\end{proposition}

\begin{proof} 
	Let $i\geq 2$.  We set $S:=k[X_1,\ldots,X_d]$ and recall that $\fn:=(X_1,\ldots,X_d)$ is generated by a regular sequence.
It is shown in \cite[2.12]{a} that the map $\Tor^S_+(\fn^j,k)\to \Tor^S_+(\fn^{j+1},k)$, induced by  $S/\fn^{j+1} \twoheadrightarrow S/\fn^{j} $, is the zero map for all $j>0$.
In view of \cite[Formula  (7) and (8)]{lev}  this property implies that $$P^N_R(t)=\frac{P^N_S(t)}{1-t(P^R_S(t)-1)}\quad(\ast)$$where
	\begin{enumerate}
	\item[i)] $N$ is any $S$-module which annihilated by $\fn^{n-1}$,
	\item[ii)]$P^A_B(t)$  is the $Poincare\acute{e}$ series of a  $B$-module $A$.
\end{enumerate} 
Now, let $N$ be any nonzero $R$-module which is annihilated
by  $\Ann(\fm)=\fm^{n-1}$.
Then $N$ is an $S$-module and  annihilated  by $\fn^{n-1}$.
It is enough to plugging $(\ast)$ into the proof of \cite[Theorem 1.1]{ram2} to see that $\beta_{i+1}(N)-\beta_{i}(N)\geq d-1$ is valid for all $n>1$ and all $i\geq1$.
Let $M$ be any nonfree $R$-module. Set $N:=\Syz_1(M)$. Then $N$ is a nonzero $R$-module and annihilated  by $\Ann(\fm)$. 
From this, $\beta_{i+1}(M)-\beta_{i}(M)\geq d-1$ is valid for all $n>1$ and all $i\geq2$.
The particular case is in Proposition \ref{d-1}.
\end{proof}

Concerning the first question,  we will present another situation (see Corollary \ref{mor1}).
The  second question is as follows:
when is	$\beta_{2}(M)-\beta_1(M)\geq d-1$? 
Recall from the proof of Proposition \ref{ob} that
$\beta_{2}(M)-\beta_1(M)\geq d-1$ provided $M$ is annihilated
by  $\Ann(\fm)$. This may be achieved:

\begin{example}\label{e1}
	Let $R =k[X_1,\ldots,X_d]/ \fn^{n}$,  $F\in\fn^{n-1}$ be nonzero and $f$ be its image in $R$.   Then $\beta_{2}(R /fR  )-\beta_1(R /fR )= d-1$.
\end{example}

\begin{proof}
	Since $X_iF\in\fn^n$ for all $i$, we have $\fm\subset(0:f)\subset\fm$, i.e., $(0:f)=\fm$. 
	It follows that
	$R^d\to R\stackrel{f }\lo R\to R/f R\to 0$ is a part of minimal free resolution of $R/f R$.
	So,  $\beta_{2}(R /fR  )-\beta_1(R /fR )= d-1$.
\end{proof}

For the dual module $M^\ast:=\Hom_R(M,R)$ we have:

\begin{corollary} 
	Let $R =k[X_1,\ldots,X_d]/ \fn^{n}$ where $d>1$ and $n>1$. If $M$ is not free, then  $\beta_{i+1}(M^{\ast})-\beta_i(M^{\ast})\geq d-1$ for all $i\geq 0$. In particular, $\beta_{1}(M^{\ast})-\beta_0(M)\geq d$.
\end{corollary}

\begin{proof}  The case $i\geq2$ is  in Proposition \ref{ob}.
 Here we deal with the case $i=0$. Let $R^n\to R^m\to M\to 0$ be the minimal presentation of $M$.
	By $D(-)$ we mean the Auslander's transpose.	By its definition,
	there is an exact sequence  $0\to M^{\ast} \to R^m\to R^n \to D(M)\to 0.$  By \cite[Proposition 2.1]{ram}
	$D(M)$ has no   free direct summand. Due to \cite[Lemma 2.6]{ram}   $M^\ast$ is not free.  We deduce that
	$\beta_{1}(M^\ast )-\beta_0(M^\ast )=\beta_{3}(D(M) )-\beta_2(D(M) )\geq d-1.$
	Similarly, the case $i=1$ follows.

By \cite[3.4]{ram} the ring is $\BNSI$.  This allow us  apply \cite[2.7]{ram} to conclude that $\beta_{0}(M^{\ast})>\beta_0(M)$.
	The first part leads us to $\beta_{1}(M^{\ast})-\beta_0(M)\geq d$.
\end{proof}

As another  question:
how good is    $\beta_{i+1}(M)-\beta_i(M)\geq d-1$? 
This is far from being sharp:

\begin{proposition}\label{sh}
	Let $R=\frac{k[X_1,\ldots,X_d]}{\fn^n}$
	where $d>1$ and $n>1 $. Let $M$ be  nonfree. Then $$\beta_{i+1}(M)-\beta_i(M)\geq {d+n-1 \choose d-1}-1>d-1  \quad\forall  i\geq 2.$$
\end{proposition}

\begin{proof}Let $n\gg 0$.
It  is shown  in the proof of \cite[Theorem 1.1]{ram2} that $\beta_{i+1}(M)-\beta_i(M)\geq \mu(\fn^n)-1>d-1$ for all $i\geq 2$.
It remains to note that $ \mu(\fn^n)={d+n-1 \choose d-1}$.
Now, let $n>1$. Then, the desired claim is a combination of Proposition \ref{ob}  and the first part.
\end{proof}

\begin{example}	\label{sh1e}
	The first item shows the above bound may be achieved
	and the second item shows that the bound $\beta_{i+1}(M)-\beta_i(M)\geq {d+n-1 \choose d-1}-1$ is not valid for $i=1$.

	\begin{enumerate}
	\item[i)] Let $R =k[X_1,X_2]/ \fn^{2}$ and look at $M:=R/(x_1)$.
	We have $(0:x_1)=(x_1,x_2)$.  The minimal presentation of $(x_1,x_2)$ is given by the matrix
	\begin{equation*}	A:= \left(
		\begin{array}{cccc}
			0	&	x_2     &  0  &  x_1\\
			x_2	&	0    & x_1    &  0
		\end{array}  \right)^{\tr}.
	\end{equation*} Since $R^4\to R^2\to \fm \to 0$ is the minimal presentation
	of $\fm$, $R^4\to R^2\to R\stackrel{x_1}\lo R\to M\to 0$ is a part of minimal free resolution of $M$. Recall that ${d+n-1 \choose d-1}={3 \choose 1}=3$.
	From this  $$\beta_{3}(M)=4=2+3-1=\beta_{2}(M)+{d+n-1 \choose d-1}-1.$$ 
	\item[ii)] Let $R =k[X_1,\ldots,X_d]/ \fn^{n}$ where $d,n>1$ and let $x:=x_1^{n-1}$. Similar to Example 	\ref{e1} we have $\beta_{2}(R /xR  )-\beta_1(R /xR )= d-1<{d+n-1 \choose d-1}-1$.
\end{enumerate}  
\end{example}
	
Let $A=\oplus_{i\geq 0} A_i$ be a  standard  graded ring over a field $A_0:=k$. We set $\fm:=\oplus_{i>0} A_i$.
There is a presentation $A=\frac{k[X_1,\ldots,X_d]}{(f_1,\ldots,f_t)}$.  By $\Ord(A)$
we mean   $\sup\{i:(f_j)\nsubseteq(X_1,\ldots,X_d)^{i+1} \}$.

\begin{corollary}\label{mor1} 
Let $(A,\fn)$ be a  standard  graded ring over $k$ such that $e:=\dim_k A_1>1$ and $\Ord(A)>1$. Let $ 1<n\leq\Ord(A)$. If $M$  is nonfree over $R:=\frac{A}{\fn^n}$ and $i\geq 2$, then     $$\beta_{i+1}(M)-\beta_{i}(M)\geq {e+n-1 \choose e-1}-1 >\dim A-1.$$	In particular, $\beta_{i+1}(\omega_R)-\beta_i(\omega_R)\geq e-1\geq\dim A-1$ for all $i\neq 1$.
\end{corollary}

\begin{proof}  We set $S:=k[X_1,\ldots,X_e]$, $I:=(f_j)$ and   $\fn:=(X_1,\ldots,X_e)$. Then
	$R=\frac{A}{\fm^n}=\frac{S}{  I+\fn^n }=\frac{S}{\fn^n}$. So, the claim follows by
	Proposition \ref{ob} and   \ref{sh}.
\end{proof}

	The next item replaces the maximal ideal in Observation  \ref{2}  with an $\fn$-primary ideal.
	
\begin{observation}
	Let $I$ be an  $(X,Y)$-primary ideal and let $R=k[X,Y]/I^{n}$    with $n>1$.
	 If $M$  is nonfree, then   $\beta_{i+1}(M)-\beta_{i}(M)\geq d-1$ for all $i\geq 2$.
\end{observation}

\subsection{General case}
\begin{lemma}\label{lem:compare}
	Let $(S,\n,k)$ be a local ring, let $n\ge2$, and set $R=S/\n^n$.
	If $\n^{n-1}N=0$, then
	\( 
	\beta_0^R(N)=\beta_0^S(N)\) and
	\(
	\beta_1^R(N)=\beta_1^S(N).
	\)
\end{lemma}
\begin{proof}
	Put $b=\dim_k(N/\n N)$ and choose a minimal surjection
	$\pi:S^b\to N$. Write $L=\ker\pi$.
	The annihilation hypothesis gives
	\(
	\n^{n-1}S^b\subseteq L\) and
	\(\n^nS^b=\n(\n^{n-1}S^b)\subseteq\n L.
	\)
	The induced minimal surjection $R^b\to N$ has kernel
	$L/\n^nS^b$. Consequently
	\(
	\beta_1^R(N)
	=\dim_k(\frac{L}{\n L+\n^nS^b})
	=\dim_k(\frac{L}{\n L})
	=\beta_1^S(N).
	\)
	Equality of the zeroth numbers follows from the definition of $b$.
\end{proof}

\begin{lemma}\label{lem:fitting}
	Let $(S,\n,k)$ be a regular local ring of dimension $d\ge1$.
	For every nonzero finite-length $S$-module $N$,
	\(
	\beta_1^S(N)-\beta_0^S(N)\ge d-1.
	\)
\end{lemma}
\begin{proof}
	Let $S^c\xrightarrow{\varphi}S^b\to N\to0$
	be a minimal presentation. Since $N$ has rank zero over the domain $S$,
	we have $c\ge b$. The ideal of maximal minors is
	$I_b(\varphi)=\Fitt_0^S(N)$. The support identity for Fitting ideals gives
	\(
	V(I_b(\varphi))=\Supp_S(N)=\{\n\},
	\)
	so $\Ht (I_b(\varphi))=d$. The classical determinantal height bound
	states that a proper ideal of maximal minors of a
	$b\times c$ matrix has height at most $c-b+1$. Thus
	$d\le c-b+1$, as required.
\end{proof}

\begin{theorem}\label{thm:allsteps}
	Let $(S,\n,k)$ be a regular local ring of dimension $d\ge2$,
	let $n\ge2$, and set $R=S/\n^n$.
	If $M$ is a nonfree finite $R$-module, then
	\(
\beta_{i+1}^R(M)-\beta_i^R(M)\ge d-1\)
 for every $i\ge1.$

\end{theorem}
\begin{proof}
	Write $\m=\n R$. In a minimal free resolution of $M$, let
	$N_i=\syz_i^R(M)$ for $i\ge1$. Minimality gives
	\(
	N_i\subseteq\m R^{\beta_{i-1}^R(M)}.
	\)
	Since $\m^n=0$, it follows that $\n^{n-1}N_i=0$.
	Also $N_i\ne0$: otherwise $M$ would have finite projective dimension;
	the Auslander--Buchsbaum formula over the Artinian ring $R$ would then
	give projective dimension zero, hence freeness.
	Now $N_i$ is a nonzero finite-length $S$-module. The two lemmas yield
$\beta_{i+1}^R(M)-\beta_i^R(M)
=\beta_1^R(N_i)-\beta_0^R(N_i)=\beta_1^S(N_i)-\beta_0^S(N_i)\ge d-1.$
\end{proof}

\begin{corollary}
	For the above rings, the exponent need only satisfy $n\ge2$.
	There is no large-exponent hypothesis. For $i=1$ the constant $d-1$
	is optimal, simultaneously for every $d,n\ge2$.
\end{corollary}
\begin{proof}
	The assertion about the exponent follows from Theorem
	\ref{thm:allsteps}. Choose $0\ne f\in\Soc(R)=\m^{n-1}$ and set
	$M=R/fR$. Then $\operatorname{ann}_R(f)=\m$. The beginning of a minimal
	resolution is
	\(	R^d\to R\xrightarrow{\,f\,}R\to M \to 0. \)
	Hence $\beta_1^R(M)=1$ and $\beta_2^R(M)=d$.
	The case $n=1$ gives a field, over which no finite nonfree modules exist.
\end{proof}

We leave to the reader to construct additional results of previous subsection to 	 $R=S/\n^n$, where $(S,\n)$ is a regular local ring.

\subsection{The infimum of $\{\beta_2^R(\omega_R):\}$}

\begin{proposition}\label{prop:inf}
	For the class of all non-Gorenstein Cohen--Macaulay local rings with
	a canonical module, let
	\(
	\mathcal B=\inf_R\beta_2^R(\omega_R).
	\)
	Then $2\le\mathcal B\le6$.  For the truncated polynomial subfamily, however,
	\[
	\inf_{d,n\ge2}\beta_2^{R_{d,n}}(\omega_{R_{d,n}})=6,
	\qquad R_{d,n}=k[x_1,\ldots,x_d]/(x_1,\ldots,x_d)^n.
	\]
\end{proposition}
\begin{proof}
	For the general lower bound, first reduce $R$ by a regular system of
	parameters. The canonical module and its minimal free resolution reduce
	along this regular sequence: the sequence is also regular on the
	canonical module, so the reduced resolution is exact and minimal.
	The Betti numbers are therefore preserved. The reduced ring is
	Artinian and is still not Gorenstein.
	
	Over an Artinian ring, put $L=\len(R)=\len(\omega_R)$ and
	$N_i=\syz_i^R(\omega_R)$, including $N_0=\omega_R$.
	The canonical module has infinite projective dimension: finite
	projective dimension would force it to be free, hence isomorphic to $R$
	by its length, making $R$ Gorenstein. Thus every $N_i$ is nonzero.
	The exact syzygy sequences give
$\len(N_{i+1})=\beta_i^R(\omega_R)L-\len(N_i).$
	By induction each $\len(N_i)$ is a positive integer multiple of $L$.
	Hence, for $i\ge1$,
$	\beta_i^R(\omega_R)
	=\frac{\len(N_i)}L+\frac{\len(N_{i+1})}L\ge2.$
	Also $\beta_1^R(\omega_R)\ge\beta_0^R(\omega_R)$.
	The example $k[x,y]/(x,y)^2$ has second number six, proving the
	unrestricted upper bound.
	
	For the assertion about the truncated family, first consider $d=2$.
	The standard Hilbert--Burch resolution of $(x,y)^n$, whose presentation
	matrix has linear entries, gives a minimal canonical-module presentation
$
	R^{n+1}\to R^n\to \omega_{R_{d,n}}\to0.
$
	Minimality over $R$ follows because its $n+1$ linear columns are
	independent over $k$; multiplying an image by $\m$ has degree at least
	two. Thus $\beta_0^R(\omega_{R_{d,n}})=n$ and $\beta_1^R(\omega_{R_{d,n}})=n+1$.
	Writing $L=\len(R)=n(n+1)/2$, the syzygy length equations give
	\[
	\len(\syz_1^R(\omega_{R_{d,n}}))=(n-1)L,
	\qquad \len(\syz_2^R(\omega_{R_{d,n}}))=2L=n(n+1).
	\]
	The linear presentation kills
	$\Soc(R^{n+1})=(\m^{n-1})^{n+1}$, which already has length $n(n+1)$.
	It is therefore the entire second syzygy. Consequently
$	\beta_2^R(\omega_{R_{d,n}})=n(n+1)\ge6.$	If $d\ge3$ and $n=2$, Proposition \ref{prop:squarezero} gives
	$\beta_2^R(\omega_{R_{d,n}})=d(d^2-1)\ge24$.
	If $d\ge3$ and $n\ge3$, then
	\[
	\beta_0^R(\omega_{R_{d,n}})=\dim_k\m^{n-1}
	=\binom{d+n-2}{d-1}\ge\binom{d+1}{2}.
	\]
	Combining $\beta_1^R(\omega_{R_{d,n}})\ge\beta_0^R(\omega_{R_{d,n}})$ with Theorem
	\ref{thm:allsteps} gives
$
	\beta_2^R(\omega_{R_{d,n}})\ge\binom{d+1}{2}+d-1\ge8.
$
	Thus the minimum over the truncated family is attained at $d=n=2$.
\end{proof}

\section{initial Betti numbers and semidualizing modules}

The ring itself and the canonical module are trivial examples of semidualizing modules.
In view of \cite[3.1]{cw} the ring $R_{d,n}$ has no nontrivial semidualizing module.
A similar claim  holds over Golod rings and rings of embedding codepth at most three, see the forthcoming  work \cite{ains}. As a
special case, and as an application, the following result deduced via  an  easy calculation  of  initial Betti numbers:

\begin{corollary}\label{none}
	Let $S$ be a  regular local ring and let $I$ be a radical and   height two Cohen-Macaulay ideal
	which is not complete-intersection. Then $S/I$ has no nontrivial semidualizing module.
\end{corollary}

\begin{proof}
	Since $R:=S/I$ is Cohen-Macaulay, $\min(R)=\Ass(R)$. Since
	$R$ is   quotient of a regular ring, it has a dualizing module $D$.
	Suppose
	on the way of contradiction that $R$ possess a nontrivial  
	semidualizing  module $C$ such that $D\ncong C\ncong R $.
	Since $I$ is radical, 
	$	R_{\fp}$ is a field   for all $\fp\in\Ass(R)$. In particular, $R$ is generically Gorenstein.
	In sum, we are in the situation of \cite[Corollary 3.8]{cw}
	to conclude that $\beta_1 (\omega_R)\geq 2\beta_0 (\omega_R).$
	Also, $\mu(I)\geq 3$. It turns out that $\beta_1 (\omega_R)=\mu(I)=\beta_0 (\omega_R)+1$, a contradiction.
\end{proof}

Monomial quotient  of a polynomial ring, may not accept  any nontrivial semidualizing module. 

\begin{example}
	Let $R:=\frac{k[x,y,z,w]}{(x^2,y^2,z^2,w^2,xy)}$. Then $R$	has no nontrivial semidualizing module.
\end{example}

\begin{proof}By Macaulay 2, the minimal free resolution of ${R}$ over $S:=k[x,y,z,w]$ is given by $0\to  S^2\stackrel{A}\lo S^7\to S^9 \to S^5\to S\to R\to 0$
	where	
	\begin{equation*}	A:= \left(
	\begin{array}{ccccccc}
	-w^2    & 0    & z^2   &0     &-y   &x     & 0 \\
	0       &-w^2  &0      & z^2  &0    &-y    & x  
	\end{array}  \right)^{\tr}.
	\end{equation*}
	By taking $\Hom_S(-,S)$  we get to a resolution $S^7\stackrel{A^t}\lo S^2\to \omega_{R}\to 0$. Apply $-\otimes_SR$ to it,
	we get to $R^7\stackrel{B}\lo  R^2\to \omega_{R}\to 0,
	$ as a presentation of $\omega_{R}$ as an $R$-module, where \begin{equation*}	B:= \left(
	\begin{array}{ccccccc}
	-\overline{w^2}    &0                 & \overline{z^2}& 0              &-\overline{y}&\overline{x}    &0\\
	0                  &-\overline{w^2}   &0             &\overline{z^2}   &0            & -\overline{y}  &\overline{x}
	\end{array}  \right)=\left(
	\begin{array}{ccccccc}
	0   &0  & 0& 0 &-\overline{y}&\overline{x}    &0\\
	0   &0  &0 &0  &0            & -\overline{y}  &\overline{x}
	\end{array}  \right).
	\end{equation*}
	In particular, the minimal presentation of $\omega_R$ is given by a matrix with $2$ rows and  $3$ columns.
	From this $\beta_1 (\omega_R)=3$ and $\beta_0 (\omega_R)=2$.
	In view of  $(\ast)$ in Corollary \ref{none}, $R$	does not accept  any nontrivial semidualizing module. 
\end{proof}

\begin{corollary} 
	Let $I$ be a monomial and $\fn$-primary ideal of $S:=k[x_1,\ldots,x_d]$. 
If $\beta_1 (\omega_{S/I})<8$, then $S/I$	has no nontrivial semidualizing module.
\end{corollary}

\begin{proof} On the way of contradiction, suppose $R:=S/I$ accepts a nontrivial semidualizing module $C$. Set $(-)^+=\Hom(-,\omega_R)$. By \cite[3.7]{cw} we have $\ell(C)=\e(C)=\e(\omega_R)=\ell(R)$. We apply this along with \cite[Lemma A.1]{csv} to see $\beta_i (C)\geq 2$. Similarly, $\beta_i (C^+)\geq 2$. Recall that
 $\beta_1 (\omega_R)=\beta_0 (C)\beta_1 (C^+)+\beta_1 (C)\beta_0 (C^+)\geq 8,$  a contradiction.
\end{proof}

\begin{example}
	Let $R:=\frac{k[x,y,z,w]}{(x^2,y^2,z^2,w^2,xyz)}$. Then $R$	has no nontrivial semidualizing module.
\end{example}

\begin{proof}By Macaulay 2, the minimal free resolution of ${R}$ over $S:=k[x,y,z,w]$ is given by $0\to  S^3\stackrel{A}\lo S^9\to S^{10} \to S^5\to S\to R\to 0$
	where	
	\begin{equation*}	A:= \left(
	\begin{array}{ccccccccc}
	 x^2    & 0    & 0  &-w     &-z   &y     & 0 & 0& 0\\
	0       &x^2   &0      & 0  &-w    &0    & -z & y& 0\\
	0       &0  &x^2      & 0 &0    &-w    & 0  & z& -y 
	\end{array}  \right)^{\tr}.
	\end{equation*}Thus $A^{\tr}\otimes_S R$  has $3$ (nonzero) rows and  $6$ columns.
	From this $\beta_1 (\omega_R)=6$ and $\beta_0 (\omega_R)=3$. By previous corollary,
	we get the claim.
\end{proof}

\section{ questions by 	Huneke  and Eisenbud-Herzog }

In this section $(S,\fm)$ is a $d$-dimensional regular local ring and $I$ is an ideal of height $c$. We assume that
	$	I = JK$ for two proper ideals $J$ and $ K$ such that $R:= S/I$ is Cohen-Macaulay.
	Huneke proved that if $c>1$ and $S$ is unramified  then $R$ is not Gorenstein. 
The following simple argument provides    another situation  for which $R$ is not Gorenstein:

\begin{proposition}\label{51}
	Let $(A,\fm,k)$ be any local ring containing $k$. Let  $K$ and $J$  be of height $\ell$ that are monomial  with respect to a regular sequence $\underline{x}$ of length $\ell>1$. Then 
	 $\type(A/KJ) \geq 2\type(A).$
	In particular, $KJ$ is not Gorenstein.
\end{proposition}

\begin{proof}Set $I:=KJ$ and  recall that  $I$  is monomial with respect to $\underline{x}:=x_1,\ldots,x_{\ell}$. Let $\{\underline{x}^{\underline{i}}:\underline{i}\in\Sigma\}$ be its generating set.
	 Let $P:=k[X_1,\ldots,X_{\ell}]$ be the polynomial ring over $k$.
By using  Hartshorne's trick (see \cite[Proposition 1]{har}), there is a flat map $f:P\to A$ defined via the assignment $X_i\mapsto x_i$.   Let $I_P:=( \underline{X}^{\underline{i}}:\underline{i}\in\Sigma)P$ be the monomial ideal of $P$. Then $I_PA=I$. By base change coefficient theorem, 
 $f:P/I_P\to A/I$ is flat. Let $\fn$ be the irrelevant ideal of $P/I_P$. 
Since $I_P$ is $\fn$-primary,  $P/I_P$ is local and zero-dimensional.   We apply this along with \cite[Proposition A.6.5]{hh} to conclude that $I_P$   is generated by pure
powers of the variables. Recall that there are proper monomial
ideals $J_P$ and $K_P$ such that $I_P=J_PK_P$. 
Let $G(-)$ be the unique minimal set of
monomial generators of a monomial ideal.
We have $G(I_P)\subset G(J_P)G(K_P) $. Since $\ell>1$
it turns out that $I_P$   is  not generated by pure
powers of the variables.
This contradiction implies that $P/I_P$
is not Gorenstein.  This, in turns, is equivalent to say that  its type is at least two.
The closed fiber of   $ f$ is $\frac{A/I}{f((\fn/I_P))A/I}=\frac{A/I}{\underline{x}/I}=\frac{A}{\underline{x}A}$. Recall that type of a module does not change under reduction
by a regular sequence, see the proof of \cite[Proposition 1.2.16]{bh}. Thus $\type(\frac{A}{\underline{x}A})=\type(A)$.
Now, in view of \cite[Proposition 1.2.16(b)]{bh}
we see $$\type(A/I)=\type(P/I_P)\type(\frac{A}{\underline{x}A})=\type(P/I_P)\type(A)\geq2\type(A)  \geq 2.$$
In particular,
$ A/I$  is not Gorenstein.
\end{proof}

\begin{example}\label{52}
	\begin{enumerate}
	\item[i)]	The assumption that  $K$ and $J$  have  height equal to  the  length of $\underline{x}$ is essential: Let $A:=k[[x,y]]$,  $\underline{x}:=x,y$, $K:=xA$ and $J:=yA$. Then $A/JK$
	is Gorenstein.
\item[ii)] The assumption $\ell>1$ is essential. Let $A:=k[[x]]$,  $\underline{x}:=x$ and $J=K=xA$. Then $A/JK$
	is Gorenstein.
\item[iii)]The  inequality  $\type(A/KJ) \geq 2\type(A) $ may be achieved:
Let $A:=k[[x,y]]$,  $\underline{x}:=x,y$, $K=J=\underline{x}$. Then $\type(A/KJ) =\dim_k\Soc(k[[x,y]]/(x,y)^2) =2=2\type(A) $.
\item[iv)] The inequality is not an equality. Let $A:=k[[x,y,z]]$,  $\underline{x}:=x,y,z$, $K=J=\underline{x}$. Then $\type(A/KJ) =3>2=2\type(A) $.
	\end{enumerate} 
\end{example}

	\begin{question}(See \cite{h})   \label{qh}
	Is	type of $R$ at least $c$?
\end{question}

\begin{remark}\label{i=jkr}  
		Huneke  remarked that the question is true if $J:=\fm$ and $K$ is $\fm$-primary. By the same argument,   we have the following  observation:
		\begin{enumerate}
		\item[i)]  If $J=\fm^t$ for some $t>0$  and $K$ is $\fm$-primary, then  $\type(R)\geq\mu(\fm^{t-1}K)\geq d$.
		\item[ii)]The bound in  item i) is   sharp: see Example \ref{52}(iii).
			\item[iii)]The bound in the item i) is not an equality: We set $S:=k[[x,y]]$, $J:=\fm$ and $K:=\fm^2$.
			Then $\Soc(R)=\frac{\fm^2}{\fm^3}$ which is of dimension $3$.
	\end{enumerate} 
\end{remark}

Here, we extend Remark \ref{i=jkr}(i) by a different argument:

\begin{observation}\label{genmax}
	Let $(A,\fm)$ be a  local ring  and $0\neq J$ be  a  primary ideal. The following holds:
	\begin{enumerate}
		\item[i)] If $d:=\dim (A)>1$, then
	$A/ \fm J$ is not Gorenstein.
		\item[ii)] If $J$ is primary to the maximal ideal, then $\type(A/ \fm J)\geq\mu(J)\geq d$.
\end{enumerate} 
\end{observation}

\begin{proof}  $i)$:
	Recall that $\depth(A/\fm J)=0$. If $J$ is primary to a non-maximal ideal, then $\dim(A/\fm J)>0$ and so $A/ \fm J$ is not Cohen-Macaulay.
	Thus, we may assume that  $J$ is primary to  the maximal ideal. Suppose on the way of contradiction that 	$A/ \fm J$ is  Gorenstein.
	In particular, any module is (totally) reflexive. Over $\BNSI$ any finitely generated reflexive module is free (see \cite[Proposition 2.7]{ram}). Combine this
	with the fact that $A/ \fm J$ is $\BNSI$ (see \cite{les}) we deduce that any finitely generated  module is free. Thus $A/ \fm J$ is regular.
	But, the local ring $A/ \fm J$ is not even reduced. This contradiction says that $A/ \fm J$ is not Gorenstein.
	
$ii)$: This follows from $\frac{J}{\fm J}\subset\Soc(A/ \fm J)$.
\end{proof}

\begin{corollary}\label{ci}
	Let $(A,\fm,k)$ be a local ring containing $k$,  
	and $J$ be a complete-intersection ideal. Let $n>1$.  
		\begin{enumerate}
		\item[i)] One has
	$ \type(A/ J^n)\geq \Ht(J)$. 	
	\item[ii)] In the case $n=2$, the equality  $ \type(A/ J^2)=\Ht(J)$ holds  provided $A$ is Gorenstein.
\end{enumerate} 
\end{corollary}

\begin{proof} 
	Let $\underline{x}:=\{x_i\}$ be a minimal generating set of $J$ which is a regular sequence.  
	Let $P:=k[X_1,\ldots,X_{d}]$ be the polynomial ring  and let $J_P$ be its monomial maximal ideal.   Recall that there is a flat map 
	$f:P/J^n_P\to A/J^n$ such that its    closed fiber is $\frac{A}{\underline{x}A}$.  By \cite[Proposition 1.2.16(b)]{bh}
	we have $$\type(A/J^n)=\type(P/J^n_P)\type(\frac{A}{\underline{x}A})\quad(\ast)$$ 
	Therefore, we may assume that $A$ is the polynomial ring and $J$
	is the maximal ideal. 
	
	i) The claim in this case is in Remark \ref{i=jkr}(ii).
	
	ii)
	Suppose  $A$ is Gorenstein and that $n=2$. Then $\type(\frac{A}{\underline{x}A})=\type(A)=1$.
	By $(\ast)$ we know $ \type(A/ J^2)= \type(P/J^2_P)=\dim P=\Ht(J)$.
\end{proof}

Here, we present a base for $\Soc(S/J^2)$:

\begin{discussion}
		Assume that $S$ is the polynomial ring and $J$ is a monomial  
	parameter ideal of full length.
	We may assume that   $J\neq \fm$. There is an $a_i$ such that $x_i^{a_i}\in J$.
	We take such an $a_i$ in a minimal way. Then $J=(x_i^{a_i}:1 \leq i\leq d)$.
	Set $x:=\prod_{i=1}^d x_i^{a_{i-1}}$. Note that $x\notin J$. Since
	$J\neq \fm$, $x\in \fm$.
	We are going to show $\fm(x_1^{a_1}x)\in J^2$.
	Let $j\geq1$. Recall that $x_1^{a_1} x_j^{a_j}\in J^2$. Since 
	$x_1^{a_1} x_j^{a_j}|x_jx_1^{a_1}x$ we deduce that $x_j(x_1^{a_1}x )\in J^2$. Hence $\fm(x_1^{a_1}x)\in J^2$. Similarly, $\fm(x_j^{a_j}x)\in J^2$. 
	By definition,  $\mathcal{B}:=\{\overline{x_1^{a_1}x},\ldots, \overline{x_d^{a_d}x}\}\subset\Soc(S/J^2)$.  By the previous corollary, $ \type(S/ J^2)= \Ht(J)$. Since   $\mathcal{B}$ is $k$-linearly independent, we see that  $\mathcal{B}$
	is a base for $\Soc(S/J^2)$.  
\end{discussion}

\begin{example}
	Let $S:=k[[x,y,z]]$ and $J:=(x,y^2,z^3)$. Then $ \type(S/ J^2)=3$.
\end{example}

\begin{corollary}\label{bd}
Let	 $K$ be an $\fm$-primary ideal of $S$ and   $t>0$. Set $d:=\dim S$. Then $\beta_{d-1}(\frac{S}{\fm^tK}) \geq 2d-1$.
\end{corollary}

 By Example
\ref{3e}
the above bound is achieved.
In general,  the inequality is not  equality. 

\begin{proof} 
	 Let $\underline{x}:=x_1,\ldots,x_{d}$ be a minimal generating set for $\fm$.
  In view of the  natural isomorphisms  $\Tor_d^S(\frac{S}{\fm^tK},k)\cong\HH_d(\underline{x},\frac{S}{\fm^tK})\cong\frac{(\fm^tK:_S\fm)}{\fm^tK}\supseteq\frac{\fm^{t-1}K}{\fm^tK}$  we see that
 $\beta_d(\frac{S}{\fm^tK})\geq\mu(\fm^{t-1}K)\geq d.$ 
We look at the exact sequence 
$0\to S^{\beta_d(S/\fm^tK)}\to S^{\beta_{d-1}(\frac{S}{\fm^tK})}\to\Syz_{d-1}(\frac{S}{\fm^tK})\to 0 $ and recall from  \cite[Corollary 1.7]{evan} that $\rank_S(\Syz_{d-1}(\frac{S}{\fm^tK}))\geq d-1$.  Therefore,   $\beta_{d-1}(\frac{S}{\fm^tK})=\beta_d(\frac{S}{\fm^tK})+ \rank_S(\Syz_{d-1}(\frac{S}{\fm^tK})) \geq d+(d-1)= 2d-1.$   
\end{proof}

 An ideal $J	$ is called generically (resp. locally) complete-intersection if $J_{\fp}$ is complete-intersection for all  $\fp\in\min(R/J)\setminus \{\fm\}$ (resp. $\fp\in\Spec(R/J)\setminus \{\fm\}$).
In particular, locally complete-intersection are generically complete-intersection.

\begin{example}Let $S$ be a regular local ring.
	\begin{enumerate}
	\item[i)] Any prime ideal is generically complete-intersection.
	\item[ii)]   Let $\{\fp_i\}_{i=1}^n$  be a family of $1$-dimensional prime ideals. Then
		$I:=\fp_1\ldots \fp_n$ is locally complete-intersection.
\end{enumerate}
\end{example}

	\begin{proposition} \label{locally}
	Let $(A,\fm,k)$ be a $d$-dimensional  local ring containing  $k$
	and $J$ be a generically complete-intersection ideal of height $\ell>1$  and assume 
	$\ell<d$. 	The following holds:
	\begin{enumerate}
		\item[i)]  The ring	$A/ J^n$ is not Gorenstein for all $n>1$.
		\item[ii)] 	$ \type(A/ J^n)\geq \Ht(J)$ provided   $A/J^n$ is Cohen-Macaulay.	
		\item[iii)] If $J^2$ is Cohen-Macaulay and $A$ is Gorenstein, then  $ \type(A/ J^2)=\Ht(J)$.
	\end{enumerate} 
\end{proposition}

\begin{proof}
We prove the claims at the same time.	Without loss of the generality we may assume that $A/J^n$ is Cohen-Macaulay.
	There is a prime ideal $\fp\in\V(J)\setminus \{\fm\}$ such that $J_{\fp}$ is  complete-intersection.
	By a complete-intersection ideal we mean an ideal generated by a regular sequence.
	In view of  \cite[Corollary 3.3.12]{bh}, 
	 $\type(\frac{ A }{J^n})\geq\type((A/J^n)_{\fp})$. Also, recall from \cite[Ex. 12.26(c)]{bh} that type of $A/J^n$
	 as a ring is the same as of the module over $A$. It remains to apply Corollary \ref{ci}.
\end{proof}

\begin{fact}\label{c2}
	Suppose $S$ is unramified regular and $I=JK$ be of height two and Cohen-Macaulay. Then $\type(R)=\mu(I)-1\geq 2$.
\end{fact}

\begin{proof}    By  \cite[Theorem 2]{h} $R$ is not Gorenstein. This allow us to apply the proof of 
	\cite[Proposition 2.7]{growth} to conclude  that   $\type(R) =\mu(I)-1$.
\end{proof}

Let us drop the unramified condition:

\begin{proposition} \label{ram}
	Let  $I$ be a  height two ideal of a regular local ring $S$ and
$I$ be a product of	two proper ideals.
	\begin{enumerate}
	\item[i)]  The ring	$R:=S/ I$ is not Gorenstein.
	\item[ii)] 	  If $R$ is Cohen-Macaulay then $\type(R)=\mu(I)-1\geq 2$.
\end{enumerate} 
\end{proposition}

\begin{proof} Without loss of the generality we may assume that $R$ is Cohen-Macaulay.
	We set $\overline{S}:=S[t]_{\fm S[t]}$ and denote its maximal ideal by $\overline{\fm}$.
	The map $f:S\to \overline{S}$ is flat and $f(\fm)\overline{S}=\overline{\fm}$. This induces a flat map
	$S/I\to \overline{S}/I\overline{S}$. By \cite[Proposition 1.2.16(b)]{bh}
	we have $\type(\overline{S}/I\overline{S})=\type(S/I)\type(\frac{\overline{S}}{\mathfrak m \overline{S}})=\type(S/I).$
	The extension does not change the regularity of $S$
	and the product property of $I$. Also, height of $I$ does not change, see \cite[Theorem A.11]{bh}. In view of \cite[Ex 1.2.25]{bh} we have$$\mu(I)=\ell (\frac {I}{\fm I})=\ell (\frac {I}{\fm I})\ell(\overline{S}/\fm \overline{S})=\ell (\frac {I\overline{S}}{\overline{\fm} I\overline{S}})=\mu(I\overline{S}).$$
	Hence,  we may  assume
	that the residue field of $S$ is infinite.  Again, similar to Huneke's argument, we choose a maximal regular
	sequence $\underline{x}:=x_1, \ldots, x_m$ on $S/I$ consisting of elements whose images in $\fm/(\fm^2 + I)$ are 
	independent.  We replace $S$ with $S_1:=S/ \underline{x} S$ which is   regular, and replace $I$
	by $I S_1$. Recall that type of a module does not change under reduction
	by a regular sequence, see \cite[proposition A.6.2]{hh}.
	Moreover, Huneke remarked that $\Ht(I)=\Ht(IS_1)$ and that $IS_1=JS_1KS_1$. If $I$  minimally generated by $\{y_1,\ldots,y_t\}$, then
	$\{\overline{y_1},\ldots,\overline{y_t}\}$ is a generating set  for $IS_1$. It follows that 
	 $\mu(IS_1)\leq \mu(I)$.  In
	 	view of the following  fact, we conclude  that $IS_1$ is not a  parameter ideal (see \cite[Proposition 1]{h}):
	 \begin{enumerate}
	 	\item[Fact A)]     Let  $A$ be a Gorenstein local ring of dimension at least two.
	 	Then an ideal   generated by a system of parameters is never the product of two
	 	proper ideals.
	 \end{enumerate} 
 Since  $IS_1$ is not primary to the maximal ideal, we deduce that $\mu(IS_1)>2$.
	 We combine this with the previous observation to conclude that   $2<\mu(IS_1)\leq \mu(I).$  
 In the light of Hilbert-Burch we see $0\to S^{\mu(I)-1}\stackrel{\phi}\lo S^{\mu(I)}\to I\to 0$ is the minimal free resolution of $I$. The presentation
$S^{\mu(I)}\stackrel{ \phi^{\tr}}\lo S^{\mu(I)-1}\to \omega_R\to 0 $   is minimal as an $S$-module. Recall that $\mu(I)-1\geq 2$ and that $I$ is generated by maximal minors of $\phi$.   So,  no column or row of $\varphi:={\phi}^{\tr}\otimes_SR$ is zero. Thus   $R^{\mu(I)}\stackrel{\varphi}\lo R^{\mu(I)-1}\to \omega_{R}\to 0$  is the minimal presentation of $\omega_{R}$ as an $R$-module. 
	Therefore, $\type(R)=\mu(I)-1\geq 2$.
\end{proof}

The above proof shows that
	Question \ref{qh} reduces to the case that  ideals are primary to the maximal ideal.
The    reduction from  $S$ to $S/ \underline{x} S$ does not  preserve radical ideals. Despite of this, we show:

\begin{proposition}\label{gj=k}
	Let $J$ be a   radical ideal  of any local ring $A$ and $n>1$.  
	The following holds:
	\begin{enumerate}
		\item[i)]
		If $\ell:=\Ht(J)>1$, then  	$A/ J^n$ is not Gorenstein.	
		\item[ii)] If $A/ J^n$ is Cohen-Macaulay, then  $ \type(\frac{ A }{ J^n})\geq \Ht(J).$ 
		\item[iii)] Assume  in addition to ii) that $A$ is regular, then  $ \type(\frac{ A }{ J^n})\geq {n+\ell-2 \choose \ell-1}.$
	\end{enumerate} 
\end{proposition} 

\begin{proof}We prove all of the claims at  same time.
	Without loss of the generality we may assume that $A/ J^n$ is Cohen-Macaulay. Let $\fp\in\min(R)$ be such that $\Ht(\fp)=\Ht(J^n)$.
	Also, $\fp\in\min(J)$. Since $J$ is radical,
	$J=\fp\cap \fp_1\cap\ldots \cap\fp_i$ where $\fp_i$ is prime and minimal over $J$.
	Recall that $\fp_i\nsubseteq \fp$.
	In particular, $\fp_i A_{\fp}=A_{\fp}$.
	Since localization behaves nicely with  respect to the intersection, we have  $J_{\fp}=\fp A_{\fp}\cap \fp_1 A_{\fp}\cap\ldots \cap\fp_iA_{\fp}= \fp A_{\fp}.$
	Since ideal-extension behaves well with respect to the product of ideals, we have  $(J^n)_{\fp}=(J_{\fp})^n=\fp^n A_{\fp}$. Recall from \cite[Ex. 12.26(c)]{bh} that type of $A/J^n$
	as a ring is the same as of the module over $A$.
	In view of  \cite[Corollary 3.3.12]{bh},  $\type(A/J^n)\geq\type((A/J^n)_{\fp})$.  
	By $k(\fp)$ we  mean  the residue field of $A_{\fp}$. 
	We put all of these together   to see that $$ \type(\frac{ A }{ J^n})\geq \dim_{k(\fp)}\Soc(\frac{A_{\fp}}{  J^n_{\fp}}) =\dim_{k(\fp)}\Soc(\frac{A_{\fp}}{  \fp^nA_{\fp}}) \geq\dim_{k(\fp)}( \frac{ \fp^{n-1} A{\fp}}{  \fp^nA_{\fp}})  \geq\mu(\fp^{n-1} A_{\fp})\geq\Ht(\fp)=\Ht(J^n).$$
	Now suppose $A$ is regular. Then $A_{\fp}$ is as well. This implies that $\mu(\fp^{n} A_{\fp})={n+\ell-1 \choose  \ell-1}$. We plug this in the previous formula
	to get 
 $ \type(\frac{ A }{ J^n})\geq \mu(\fp^{n-1} A_{\fp})={n+\ell-2 \choose \ell-1}.$ 
	This is what we want to prove.
\end{proof} 

\begin{corollary}\label{j=kr}
	Question \ref{qh} is true if  $J:=\fp$ is prime and $K$ is $\fp$-primary.
\end{corollary}	

\begin{proof}Recall that  $ \type(S/\fp K)\geq \dim_{k(\fp)}\Soc(\frac{S_{\fp}}{K_{\fp}\fp S_{\fp}})=  
	\dim_{k(\fp)} (\frac{(\fp K_{\fp} :_{S_{\fp}}\fp)}{\fp K_{\fp} })
	\geq\dim_{k(\fp)} (\frac{K_{\fp}}{\fp K_{\fp}})= \mu(K_{\fp})\geq\dim(S_{\fp})=\Ht(\fp)=\Ht(\fp K).$ 
\end{proof}

By $\fp^{(n)}$ we mean the $n$-th symbolic power.

\begin{corollary}
	Let $\fp$ be an $1$-dimensional prime ideal of a local ring $A$ such that $\fp^2=\fp^{(2)}$.  Then $\type(A/\fp^2)\geq \Ht(\fp)$.	In particular, $A_{\fp}$ is regular  when the equality holds.
\end{corollary}

\begin{proof}
	We have $\HH^0_{\fm}(\frac{A}{ \fp^2})=\frac{\fp^{(2)}}{\fp^2}=0$. Since $\dim(\frac{A}{ \fp^2})=1$, $A/\fp ^2$ is Cohen-Macaulay.  By Proposition \ref{gj=k} we have
	$ \type(A/\fp ^2)\geq  \mu(\fp A_{\fp})\geq\dim(A_{\fp}).$ 
	Suppose now that the equality holds. This yields that $\mu(\fp A_{\fp})=\dim(A_{\fp})$.
	Thus,   $A_{\fp}$ is regular.
\end{proof}

\begin{corollary}
	Let $\fp$ be an $1$-dimensional prime ideal of a local ring $A$ such that $\fp^n=\fp^{(n)}$ for some $n>1$.  Then $\type(A/\fp^n)\geq   \mu(\fp^{n-1} A_{\fp})\geq\Ht(\fp)$.	
\end{corollary}

Let $(R, \fm)$ be a Cohen-Macaulay local ring. In general, 
$\mu(\fm) - \dim R + 1 \leq \e(R)$. If the equality holds we say $R$ is of minimal multiplicity.

\begin{proposition}\label{min}
	Let   $(A,\fn,k)$ be a   Cohen-Macaulay   local ring and $I= JK$ for two proper ideals $J, K$    
	such that $ A/ I$ is Cohen-Macaulay and
	of minimal   multiplicity. Then $ \type( A/ I)\geq \Ht(I).$ The equality implies that  $A$ is regular.
\end{proposition}	

\begin{proof}Set $R:= A/ I $, $d:=\dim( A)$ and $c:=\Ht(I)$.  From the Cohen-Macaulay property of $A$ and $R$ we deduce that
 $\depth(R)=\dim (R)=\dim( A)-\Ht(I)=d-c:=\ell.$ 
	By the proof of Proposition \ref{ram} we may and do assume that
	$k$ is infinite.
	In this case by a theorem of Abhyankar, there is a regular sequence $\underline{x}:=x_1,\ldots,x_{\ell}$ of $(R,\fm,k)$
	such that $\fm^2=\underline{x}\fm$. Since $I=JK\subset \fn^2$ we deduce that $$\mu(\fn)=\dim_k(\frac{\fn}{\fn^2})=\dim_k(\frac{\fn}{\fn^2+I})=\dim_k(\frac{\fn/ I}{(\fn/I)^2})=\mu(\fm)\quad(\ast)$$
	Since 	 
		$\type(R)=\type(R/\underline{x}R)$ and that
	 $$\mu(\frac{\fm}{\underline{x}R})\geq \mu(\fm)-\ell\stackrel{(\ast)}=\mu(\fn)-\ell\geq  d-\ell=c=
		\Ht(I) \quad(+)$$
things reduced to showing that $\type(R/\underline{x}R)\geq \mu(\frac{\fm}{\underline{x}R})$.
We may assume that $\fm^2=0$ and we are going to show that $\type(R)\geq \mu(\fm)$.
 Then
$\fm\subset(0:_R\fm)\subset \fm$. Thus, $ \type(R)=\dim_{k}\Soc(R)=  
 \mu(\fm).$ This completes the proof of first claim.
Suppose the equality $ \type(R)= \Ht(I)$ holds. In view of $(
	+)$ we see that
	$\mu(\fn)=d=\dim(A)$. This, in turns, is equivalent with the regularity of $A$.
\end{proof}

\begin{remark}
	In Proposition \ref{min}   the  condition $I= JK$ can be replaced by the weaker assumption $I\subset \fn^2$.
\end{remark}

The following extends a result of Herzog et al. from polynomial rings to Cohen-Macaulay rings.

	\begin{proposition} \label{her}
Let $(A,\fm,k)$ be a $d$-dimensional Cohen-Macaulay local ring containing $k$ and  $d>2$. Let  $K$ and $J$  be of height $d$ that are monomial  with respect to a regular sequence $\underline{x}:=x_1,\ldots,x_{d}$. 	The following holds:
\begin{enumerate}
\item[i)]  If $A$ is Gorenstein, then $\type(A/JK)\geq 3$.
\item[ii)] If $A$ is not Gorenstein, then $\type(A/JK)\geq 3\type(A)\geq 6$.
\end{enumerate}
	\end{proposition}
	
\begin{proof}
Set $I:=JK$.    Let $\{\underline{x}^{\underline{i}}:\underline{i}\in\Sigma\}$ be the generating set for $I$.
Let $P:=k[X_1,\ldots,X_{d}]$ be the polynomial ring  and let $I_P:=( \underline{X}^{\underline{i}}:\underline{i}\in\Sigma)P$ be a monomial ideal. Then $I_PA=I$. Recall that there is a flat map 
$f:P/I_P\to A/I$ is flat, and  there are proper monomial
		ideals $J_P$ and $K_P$ such that $I_P=J_PK_P$. 

i)  
		The closed fiber of   $ f$ is $\frac{A}{\underline{x}A}$.
		 This is of type one. By \cite[Proposition 1.2.16(b)]{bh}
		we have $\type(A/I)=\type(P/I_P)\type(\frac{A}{\underline{x}A})=\type(P/I_P)\geq 3,$  
		where the last inequality is in \cite[Theorem 5.1]{hhm}.
	
ii)  The closed fiber of   $ f$ is $\frac{A}{\underline{x}A}$. This is not Gorenstein and consequently its  type is at least two. 
		We have
		$\type(A/I)=\type(P/I_P)\type(\frac{A}{\underline{x}A})\geq 3\times 2=6.$ 	
	\end{proof}

\begin{acknowledgement}
	We check some of the examples by the help of Macaulay2. Also,
	I thank    Saeed Nasseh. 
 
\end{acknowledgement}

%%%%%%%%%%%%%%%%%%%%%%%%%%%%%%%%%%%%%%%%

\end{document}